\documentclass[12pt]{article}

\usepackage{amsfonts,amsmath,amsthm,amssymb,amscd}
\usepackage[dvips]{graphicx}
\usepackage{mathrsfs}

\theoremstyle{definition}
\newtheorem{theorem}{Theorem}

\newtheorem{proposition}[theorem]{Proposition}
\newtheorem{corollary}[theorem]{Corollary}

\numberwithin{equation}{section}
\numberwithin{theorem}{section}

\begin{document}

\begin{center}
{\bf{\Large Differential equations involving \\ cubic theta functions and Eisenstein series }}
\end{center}

\begin{center}
By Kazuhide Matsuda
\end{center}

\begin{center}
Faculty of Fundamental Science, National Institute of Technology, Niihama College,\\
7-1 Yagumo-chou, Niihama, Ehime, Japan, 792-8580. \\
E-mail: matsuda@sci.niihama-nct.ac.jp  \\
Fax: 81-0897-37-7809 
\end{center}

\noindent
{\bf Abstract}
In this paper, 
we derive systems of ordinary differential equations (ODEs) satisfied by modular forms of level three, 
which are level three versions of Ramanujan's system of ODEs satisfied by the classical Eisenstein series.  

\section{Introduction}
\label{intro}
Throughout this paper, let $\mathbb{N}_0,$ and $\mathbb{N}$ 
denote the sets of nonnegative integers and positive integers. 
For the positive integers $j,k,$ and $n\in\mathbb{N},$ 
$d_{j,k}(n)$ denotes the number of positive divisors $d$ of $n$ such that $d\equiv j \,\,\mathrm{mod} \,k.$ 
Moreover, 
for $k$ and $n\in\mathbb{N},$ 
$\sigma_k(n)$ is the sum of the $k$-th power of the positive divisors of $n,$ 
and 
$d_{j,k}(n)=\sigma_k(n)=0$ for $n\in\mathbb{Q}\setminus\mathbb{N}_0.$ 
For each $n\in\mathbb{N},$ set 
\begin{equation*}
\left(\frac{n}{3}\right)
=
\begin{cases}
+1, &\text{if $ n \equiv 1 \,\,\mathrm{mod} \,3,$}  \\
-1, &\text{if $n\equiv -1 \,\,\mathrm{mod} \,3,$}  \\
0, &\text{if $n\equiv 0 \,\,\mathrm{mod} \,3.$}
\end{cases}
\end{equation*}
\par
The {\it upper half plane} $\mathbb{H}^2$ is defined by 
$
\mathbb{H}^2=
\{
\tau\in\mathbb{C} \,\, | \,\, \Im \tau>0
\}. 
$
Throughout this paper, set $q=\exp(2\pi i  \tau)$ and define the {\it Dedekind eta function} by 
\begin{equation*}
\displaystyle
\eta(\tau)=q^{\frac{1}{24}} \prod_{n=1}^{\infty} (1-q^n)=q^{\frac{1}{24}} (q; q)_{\infty}. 
\end{equation*}
The Eisenstein series $E_2, E_4$, and $E_6$ are defined by 
\begin{align*}
E_2(q)=E_2(\tau)&:=1-24\sum_{n=1}^{\infty} \sigma_1(n) q^n, \,\,
E_4(q)=E_4(\tau):=1+240\sum_{n=1}^{\infty} \sigma_3(n) q^n, \\
E_6(q)=E_4(\tau)&:=1-504\sum_{n=1}^{\infty} \sigma_5(n) q^n, \,\,\, \sigma_k(n)=\sum_{d|n} d^k \,\, \mathrm{for} \,\,k\in\mathbb{N}. 
\end{align*}
\par
Many research efforts have been devoted to ordinary differential equations (ODEs) satisfied by modular forms. 
Classical examples include 
Ramanujan's coupled ODEs for Eisenstein series $E_2, E_4, E_6;$  
Pol and Rankin's fourth-order ODE satisfied by $\Delta=\eta^{24}(\tau);$ 
Jacobi's third-order ODE satisfied by the theta null functions $\vartheta_2, \vartheta_3, \vartheta_4,$ which are defined by equation (\ref{eqn:theta-null}).  
The proofs of these respective ODEs can be found in the papers by Berndt \cite[pp.92]{Berndt}, Rankin \cite{Rankin}, and Jacobi \cite{Jacobi}. 
\par
Halphen subsequently \cite{Halphen} rewrote Jacobi's ODE as a nonliear dynamical system: 
\begin{equation*}
X^{\prime}+Y^{\prime}=2XY, \,
Y^{\prime}+Z^{\prime}=2YZ, \,
Z^{\prime}+X^{\prime}=2ZX. \,
\end{equation*} 
In particular, ODEs of the quadratic type are known as Halphen-type systems. 
\par
Recently, 
Ohyama \cite{Ohyama1, Ohyama2} 
reconsidered Jacobi's ODE by taking into account Picard-Fuchs equations of elliptic modular surfaces, and 
following Jacobi's idea, derived a Halphen-type system satisfied by modular forms of level three. 
Following Ohyama, Mano \cite{Mano} derived ODEs satisfied by modular forms of level five. 
\par
Ramanujan's system of ODEs is expressed as follows: 
\begin{equation}
\label{eqn:Ramanujan-ODE}
q\frac{d E_2}{dq}=\frac{(E_2)^2-E_4  }{12  }, \,\,
q\frac{d E_4}{dq}=\frac{E_2 E_4-E_6  }{3  }, \,\,
q\frac{d E_6}{dq}=\frac{E_2 E_6-(E_4)^2  }{2  }, 
\end{equation}
which is equivalent to Chazy's third-order nonlinear ODE,
\begin{equation*}
y^{\prime \prime \prime}=
2y y^{\prime \prime}-3(y^{\prime})^2, \,\,y=\pi i E_2(\tau). 
\end{equation*}
Ohyama \cite{Ohyama0} showed that 
Halphen's differential field is an extension of Ramanujan's differential field, whose Galois group is the symmetric group, $S_3.$ 
\par
Ramamani \cite{Ramamani} introduced 
\begin{equation*}
\mathcal{P}=
1-8
\sum_{n=1}^{\infty}
\frac{(-1)^n n q^n}{1-q^n},  \,\,
\tilde{\mathcal{P}}=
1+24\sum_{n=1}^{\infty} \frac{n q^n}{1+q^n}, \,\,
\mathcal{Q}=1+16 \sum_{n=1}^{\infty} \frac{(-1)^n n^3 q^n   }{1-q^n  }, 
\end{equation*}
and derived a system of ODEs satisfied by the $\mathcal{P}, \tilde{\mathcal{P}}, \mathcal{Q}, $ modular forms of $\Gamma_0 (2).$
Ablowitz \cite{Ablowitz} et al. showed that 
this system is equivalent to the following third-order non-linear ODE found by Bureau \cite{Bureau}, 
\begin{equation*}
y^{\prime \prime \prime}=
2 y y^{\prime \prime}-(y^{\prime})^2
+2
\frac
{
(y^{\prime \prime}-y y^{\prime}  )^2
}
{
2 y^{\prime}-y^2
}, \,\, y=\mathcal{P}(\tau),
\end{equation*}
and that it is equivalent to a Halphen-type system.
Maier \cite{Maier} generalized these results to the Hecke group $\Gamma_0(N) \,\,(N=2,3,4). $
\par
In \cite{Huber}, Huber derived systems of ODEs satisfied by the cubic theta functions,
\begin{align*}
a(q)=&\sum_{m,n\in\mathbb{Z}} q^{m^2+mn+n^2}, \,\,
b(q)=\sum_{m,n\in\mathbb{Z}} \omega^{n-m} q^{m^2+mn+n^2}, \\
c(q)=&\sum_{m,n\in\mathbb{Z}} q^{(n+\frac13)^2+(n+\frac13)(m+\frac13)+(m+\frac13)^2}, \,\,\omega=e^{\frac{2\pi i}{3}}, \,\,|q|<1.
\end{align*}
The systems of ODEs are given by 
\begin{equation}
\label{eqn:Huber-ODE-1}
q\frac{d}{dq}a=\frac{a\mathscr{P}-b^3}{3}, \,\,
q\frac{d}{dq}\mathscr{P}=\frac{\mathscr{P}^2-ab^3}{3}, \,\,
q\frac{d}{dq}b^3=\mathscr{P}b^3-a^2 b^3,\,\,
\end{equation}
and
\begin{equation}
\label{eqn:Huber-ODE-2}
q\frac{d}{dq}a=\frac{c^3-a\mathcal{P}}{3}, \,\,
q\frac{d}{dq}\mathcal{P}=\frac{ac^3-\mathcal{P}^2}{3}, \,\,
q\frac{d}{dq}c^3=c^3a^2-\mathcal{P}c^3,
\end{equation}
where
\begin{equation*}
\mathscr{P}(q)=1-6\sum_{n=1}^{\infty} \frac{\cos(2n\pi/3)nq^n}{1-q^n}, \,\,
\mathcal{P}(q)=9\sum_{n=1}^{\infty}\frac{n(q^n+q^{2n})}{1-q^{3n}}. 
\end{equation*}
\par
The aim of the research presented in this paper is to derive systems of ODEs satisfied by $a(q)$ by means of 
Farkas and Kra's theory of 
theta functions with rational characteristics.

\par
Our main theorems are as follows.

\begin{theorem}
\label{thm-level3-E_2(q)}
{\it
For $q\in\mathbb{C}$ with $|q|<1,$ 
set 
\begin{align}
P(q)=&a(q)=\sum_{m,n\in\mathbb{Z}} q^{m^2+mn+n^2},  \,\,
Q(q)=E_2(q)=1-24\sum_{n=1}^{\infty} \sigma_1(n) q^n,   \notag  \\
R(q)=&b^3(q)=
\frac
{
(q;q)_{\infty}^9
}
{
(q^3;q^3)_{\infty}^3
}
=
1-9
\sum_{n=1}^{\infty}q^n \left( \sum_{d|n} d^2 \left( \frac{d}{3} \right)  \right).  \label{eqn:main:pro-series(1)}
\end{align}
Then, we have 
\begin{equation}
\label{eqn-level3-E_2(q)}
q\frac{d}{dq}P=\frac{3P^3+PQ-4R}{12},  \,\,
q\frac{d}{dq}Q=\frac{-9P^4+8PR+Q^2}{12}, \,\,
q\frac{d}{dq}R=\frac{-P^2R+QR}{4}. 
\end{equation}
}
\end{theorem}


\begin{theorem}
\label{thm-level3-E_2(q^3)}
{\it
For $q\in\mathbb{C}$ with $|q|<1,$ set
\begin{align}
P(q)=&a(q)=\sum_{m,n\in\mathbb{Z}} q^{m^2+mn+n^2},  \,\,
Q(q)=E_2(q^3)=1-24\sum_{n=1}^{\infty} \sigma_1(n) q^{3n}, \notag  \\
R(q)=&c^3(q)=27
q\frac{(q^3;q^3)_{\infty}^9}{(q;q)_{\infty}^3}
=27\sum_{n=1}^{\infty} q^n \left( \sum_{d|n} d^2 \left( \frac{n/d}{3} \right)  \right). 
\label{eqn:main:pro-series(2)}
\end{align}
Then, we have 
\begin{equation}
\label{eqn-level3-E_2(q^3)}
q\frac{d}{dq}P=\frac{-3P^3+3PQ+4R}{12},  \,\,
q\frac{d}{dq}Q=\frac{-9P^4+8PR+9Q^2}{36}, \,\,
q\frac{d}{dq}R=\frac{P^2R+3QR}{4}. 
\end{equation}
}
\end{theorem}

\par
Section \ref{sec:properties} reviews Farkas and Kra's theory of theta functions with rational characteristics. 
Section \ref{sec:theta-functions} treats some theta functional formulas. In particular, 
we prove equations (\ref{eqn:main:pro-series(1)}) and (\ref{eqn:main:pro-series(2)}). 
Section \ref{sec:proof-preliminary} describes preliminary results. 
Section \ref{sec:proof:level3-E_2(q)} proves Theorem \ref{thm-level3-E_2(q)}. 
In particular, our method recovers Ramanujan's system of ODEs (\ref{eqn:Ramanujan-ODE}).   
Section \ref{sec:proof:level3-E_2(q^3)} proves Theorem \ref{thm-level3-E_2(q^3)}. 
Section \ref{sec:applications} 
expresses $a^2(q), a^3(q),a^4(q), a^5(q)$ and $a^6(q)$ in terms of modular forms and divisor functions.    
Section \ref{sec:Ramanujan-a(q)} shows Ramanujan's identity of $a(q),$ for selected cases that express $a(q)$ by Dedekind's eta functions. 
Section \ref{sec:more-pro-series-level3} derives more product-series identities. 

\subsubsection*{Remark 1}
We note the properties of 
the cubic theta functions;
\begin{align}
a(q)=&a(\tau)=1+6\sum_{n=1}^{\infty} (d_{1,3}(n)-d_{2,3}(n)) q^n,  \notag \\
b^3(q)=&b^3(\tau)=\frac
{
(q;q)_{\infty}^9
}
{
(q^3;q^3)_{\infty}^3
}, \,\,
c^3(q)=c^3(\tau)=27q\frac{(q^3;q^3)_{\infty}^9}{(q;q)_{\infty}^3},   \notag \\
a^3(q)=&b^3(q)+c^3(q), \,\,\,q=\exp(2 \pi i \tau).  \label{eqn:Ramanujan-cubic-intro}
\end{align}
For the proof, readers are referred to the books by Berndt \cite[pp. 79]{Berndt}, Dickson \cite[pp. 68]{Dickson} 
and the paper by Borwein et al. \cite{BBG}.

\subsubsection*{Remark 2}
Let us define the operators
\begin{equation*}
\theta:=q\dfrac{d}{dq}=\frac{1}{2\pi i} \frac{d}{d\tau}, \,\,
\partial:=\partial_k=12 \theta -kE_2(q), \,\,(k=1,2,3,\ldots).
\end{equation*}
The properties of the operators $\theta$ and $\partial$ 
can be found in Lang's book \cite[pp. 159-175]{Lang}. 
The first ODE of Theorem \ref{thm-level3-E_2(q)} implies 
\begin{equation*}
\partial_1 a(q)=3 a^3(q)-4b^3(q).
\end{equation*}
The first ODE of Theorem \ref{thm-level3-E_2(q^3)} was proved in Cooper \cite[pp. 263]{Cooper}. 

\subsubsection*{Remark 3}
Equation (\ref{eqn:Huber-ODE-1}) and Theorem \ref{thm-level3-E_2(q)}
implies that the following differential fields, $(\quad, q d/dq),$ are equal:
\begin{align*}
\mathbb{C}\langle a(q), \mathscr{P}(q), b^3(q) \rangle=&\mathbb{C}\langle a(q), \mathscr{P}(q) \rangle
=\mathbb{C}\langle a(q), b^3(q) \rangle=\mathbb{C}\langle \mathscr{P}(q), b^3(q) \rangle \\
=&\mathbb{C}\langle a(q), E_2(q) \rangle=\mathbb{C}\langle a(q), E_2(q^3)\rangle=\mathbb{C}\langle a(q), E_2(q), b^3(q) \rangle. 
\end{align*}
Moreover, Theorem \ref{thm-E4-E6-level3-q} expresses $E_4(q)$ and $ E_6(q)$ by $a(q)$ and $b^3(q),$ 
which implies that 
\begin{equation*}
\mathbb{C}(E_2(q), E_4(q), E_6(q) )\subset \mathbb{C}(a(q), E_2(q), b^3(q) ). 
\end{equation*}

\subsection*{Acknowledgments}
This work was supported by JSPS KAKENHI Grant Number JP17K14213. 
We are grateful to Professor S. Nishioka for his useful suggestions. 
Moreover, we thank the referee for recommending various improvements to the paper.

\section{Properties of the theta functions}
\label{sec:properties}

\subsection{Definitions}
Following the work of Farkas and Kra \cite{Farkas-Kra}, 
we introduce the {\it theta function with characteristics,} 
which is defined by 
\begin{align*}
\theta 
\left[
\begin{array}{c}
\epsilon \\
\epsilon^{\prime}
\end{array}
\right] (\zeta, \tau) 
=
\theta 
\left[
\begin{array}{c}
\epsilon \\
\epsilon^{\prime}
\end{array}
\right] (\zeta) 
:=&\sum_{n\in\mathbb{Z}} \exp
\left(2\pi i\left[ \frac12\left(n+\frac{\epsilon}{2}\right)^2 \tau+\left(n+\frac{\epsilon}{2}\right)\left(\zeta+\frac{\epsilon^{\prime}}{2}\right) \right] \right), 
\end{align*}
where $\epsilon, \epsilon^{\prime}\in\mathbb{R}, \, \zeta\in\mathbb{C},$ and $\tau\in\mathbb{H}^{2}.$ 
The {\it theta constants} are given by 
\begin{equation*}
\theta 
\left[
\begin{array}{c}
\epsilon \\
\epsilon^{\prime}
\end{array}
\right]
:=
\theta 
\left[
\begin{array}{c}
\epsilon \\
\epsilon^{\prime}
\end{array}
\right] (0, \tau).
\end{equation*}
In particular, note that 
\begin{equation}
\label{eqn:theta-null}
\vartheta_2=
\theta 
\left[
\begin{array}{c}
1 \\
0
\end{array}
\right], \,\,
\vartheta_3=
\theta 
\left[
\begin{array}{c}
0 \\
0
\end{array}
\right], \,\,
\vartheta_4=
\theta 
\left[
\begin{array}{c}
0 \\
1
\end{array}
\right]. 
\end{equation}
Furthermore, 
we denote the derivative coefficients of the theta function by 
\begin{equation*}
\theta^{\prime} 
\left[
\begin{array}{c}
\epsilon \\
\epsilon^{\prime}
\end{array}
\right]
:=\left.
\frac{\partial}{\partial \zeta} 
\theta 
\left[
\begin{array}{c}
\epsilon \\
\epsilon^{\prime}
\end{array}
\right] (\zeta, \tau)
\right|_{\zeta=0}, 
\,
\theta^{\prime \prime} 
\left[
\begin{array}{c}
\epsilon \\
\epsilon^{\prime}
\end{array}
\right]
:=\left.
\frac{\partial^2 }{\partial \zeta^2} 
\theta 
\left[
\begin{array}{c}
\epsilon \\
\epsilon^{\prime}
\end{array}
\right] (\zeta, \tau)
\right|_{\zeta=0},
\end{equation*}
and
\begin{equation*}
\theta^{\prime \prime \prime} 
\left[
\begin{array}{c}
\epsilon \\
\epsilon^{\prime}
\end{array}
\right]
:=\left.
\frac{\partial^3 }{\partial \zeta^3} 
\theta 
\left[
\begin{array}{c}
\epsilon \\
\epsilon^{\prime}
\end{array}
\right] (\zeta, \tau)
\right|_{\zeta=0},\,\,
\theta^{(n)} 
\left[
\begin{array}{c}
\epsilon \\
\epsilon^{\prime}
\end{array}
\right]
:=\left.
\frac{\partial^n}{\partial \zeta^n} 
\theta 
\left[
\begin{array}{c}
\epsilon \\
\epsilon^{\prime}
\end{array}
\right] (\zeta, \tau)
\right|_{\zeta=0}, \,\,\,(n=1,2,3,4,\ldots). 
\end{equation*}
In particular, Jacobi's derivative formula is given by 
\begin{equation}
\label{eqn:Jacobi-derivative}
\theta^{\prime} 
\left[
\begin{array}{c}
1 \\
1
\end{array}
\right] 
=
-\pi 
\theta
\left[
\begin{array}{c}
0 \\
0
\end{array}
\right] 
\theta
\left[
\begin{array}{c}
1 \\
0
\end{array}
\right] 
\theta
\left[
\begin{array}{c}
0 \\
1
\end{array}
\right].  
\end{equation}

\subsection{Basic properties}
We first note that 
for $m,n\in\mathbb{Z},$ 
\begin{equation}
\label{eqn:integer-char}
\theta 
\left[
\begin{array}{c}
\epsilon \\
\epsilon^{\prime}
\end{array}
\right] (\zeta+n+m\tau, \tau) =
\exp(2\pi i)\left[\frac{n\epsilon-m\epsilon^{\prime}}{2}-m\zeta-\frac{m^2\tau}{2}\right]
\theta 
\left[
\begin{array}{c}
\epsilon \\
\epsilon^{\prime}
\end{array}
\right] (\zeta,\tau),
\end{equation}
and 
\begin{equation}
\theta 
\left[
\begin{array}{c}
\epsilon +2m\\
\epsilon^{\prime}+2n
\end{array}
\right] 
(\zeta,\tau)
=\exp(\pi i \epsilon n)
\theta 
\left[
\begin{array}{c}
\epsilon \\
\epsilon^{\prime}
\end{array}
\right] 
(\zeta,\tau).
\end{equation}
Furthermore, 
it is easy to see that 
\begin{equation*}
\theta 
\left[
\begin{array}{c}
-\epsilon \\
-\epsilon^{\prime}
\end{array}
\right] (\zeta,\tau)
=
\theta 
\left[
\begin{array}{c}
\epsilon \\
\epsilon^{\prime}
\end{array}
\right] (-\zeta,\tau)
\,\,
\mathrm{and}
\,\,
\theta^{\prime} 
\left[
\begin{array}{c}
-\epsilon \\
-\epsilon^{\prime}
\end{array}
\right] (\zeta,\tau)
=
-
\theta^{\prime} 
\left[
\begin{array}{c}
\epsilon \\
\epsilon^{\prime}
\end{array}
\right] (-\zeta,\tau).
\end{equation*}
\par
For $m,n\in\mathbb{R},$ 
we see that 
\begin{align}
\label{eqn:real-char}
&\theta 
\left[
\begin{array}{c}
\epsilon \\
\epsilon^{\prime}
\end{array}
\right] \left(\zeta+\frac{n+m\tau}{2}, \tau\right)   \notag\\
&=
\exp(2\pi i)\left[
-\frac{m\zeta}{2}-\frac{m^2\tau}{8}-\frac{m(\epsilon^{\prime}+n)}{4}
\right]
\theta 
\left[
\begin{array}{c}
\epsilon+m \\
\epsilon^{\prime}+n
\end{array}
\right] 
(\zeta,\tau). 
\end{align}
We note that 
$\theta 
\left[
\begin{array}{c}
\epsilon \\
\epsilon^{\prime}
\end{array}
\right] \left(\zeta, \tau\right)$ has only one zero in the fundamental parallelogram, 
which is given by 
$$
\zeta=\frac{1-\epsilon}{2}\tau+\frac{1-\epsilon^{\prime}}{2}. 
$$

\subsection{Jacobi's triple product identity}
All the theta functions have infinite product expansions, which are given by 
\begin{align}
\theta 
\left[
\begin{array}{c}
\epsilon \\
\epsilon^{\prime}
\end{array}
\right] (\zeta, \tau) &=\exp\left(\frac{\pi i \epsilon \epsilon^{\prime}}{2}\right) x^{\frac{\epsilon^2}{4}} z^{\frac{\epsilon}{2}}    \notag  \\
                           &\quad 
                           \displaystyle \times\prod_{n=1}^{\infty}(1-x^{2n})(1+e^{\pi i \epsilon^{\prime}} x^{2n-1+\epsilon} z)(1+e^{-\pi i \epsilon^{\prime}} x^{2n-1-\epsilon}/z),  \label{eqn:Jacobi-triple}
\end{align}
where $x=\exp(\pi i \tau)$ and $z=\exp(2\pi i \zeta).$ 
Therefore, it follows from Jacobi's derivative formula (\ref{eqn:Jacobi-derivative}) that 
\begin{equation*}
\label{eqn:Jacobi}
\theta^{\prime} 
\left[
\begin{array}{c}
1 \\
1
\end{array}
\right](0,\tau) 
=
-2\pi 
q^{\frac18}
\prod_{n=1}^{\infty}(1-q^n)^3, \,\,q=\exp(2\pi i \tau). 
\end{equation*}

\subsection{Spaces of $N$-th order $\theta$-functions}

Based on the results of Farkas and Kra \cite{Farkas-Kra}, 
we define 
$\mathcal{F}_{N}\left[
\begin{array}{c}
\epsilon \\
\epsilon^{\prime}
\end{array}
\right] $ to be the set of entire functions $f$ that satisfy the two functional equations, 
$$
f(\zeta+1)=\exp(\pi i \epsilon) \,\,f(\zeta),
$$
and 
$$
f(\zeta+\tau)=\exp(-\pi i)[\epsilon^{\prime}+2N\zeta+N\tau] \,\,f(\zeta), \quad \zeta\in\mathbb{C},  \,\,\tau \in\mathbb{H}^2,
$$ 
where 
$N$ is a positive integer and 
$\left[
\begin{array}{c}
\epsilon \\
\epsilon^{\prime}
\end{array}
\right] \in\mathbb{R}^2.$ 
This set of functions is referred to as the space of {\it $N$-th order $\theta$-functions with characteristic}
$\left[
\begin{array}{c}
\epsilon \\
\epsilon^{\prime}
\end{array}
\right]. $ 
Note that 
$$
\dim \mathcal{F}_{N}\left[
\begin{array}{c}
\epsilon \\
\epsilon^{\prime}
\end{array}
\right] =N.
$$
The proof of this space was reported by Farkas and Kra \cite[pp.133]{Farkas-Kra}.

\subsection{The heat equation}
The theta function satisfies the following heat equation: 
\begin{equation}
\label{eqn:heat}
\frac{\partial^2}{\partial \zeta^2}
\theta
\left[
\begin{array}{c}
\epsilon \\
\epsilon^{\prime}
\end{array}
\right](\zeta,\tau)
=
4\pi i
\frac{\partial}{\partial \tau}
\theta
\left[
\begin{array}{c}
\epsilon \\
\epsilon^{\prime}
\end{array}
\right](\zeta,\tau). 
\end{equation}

\section{Some theta functional formulas}
\label{sec:theta-functions}

We introduce Weierstrass' $\wp$-function and $\sigma$-function:
\begin{align*}
\wp(z;\omega_1, \omega_2)=\wp(z)
=&
\frac{1}{z^2}+\sum_
{
\tiny{
\begin{matrix}
(m,n)\in\mathbb{Z}^2 \\ 
(m,n)\neq(0,0)
\end{matrix}
}
}
\left(
\frac{1}{(z-m\omega_1-n\omega_2)^2}-\frac{1}{(m\omega_1+n\omega_2)^2}
\right),   \\
\sigma(z;\omega_1, \omega_2)=\sigma(z)
=&
z
\prod_
{
\tiny{
\begin{matrix}
(m,n)\in\mathbb{Z}^2 \\ 
(m,n)\neq(0,0)
\end{matrix}
}
}\left( 1-\frac{z}{m\omega_1+n\omega_2}  \right)
\exp
\left(
\frac{z}{m\omega_1+n\omega_2}
+
\frac{z^2}{2(m\omega_1+n\omega_2)^2}
\right),
\end{align*}
where $z,\omega_1,\omega_2$ are complex numbers with $\omega_2/ \omega_1 \not\in\mathbb{R}.$
\par
From Whittaker and Watson \cite[pp. 437, 459]{WW}, we recall the following formulas:
\begin{equation}
\frac{\sigma(4z) }{\sigma^4(z)}=-\wp^{\prime}(z),  \,\,
\frac{\sigma(3z) }{\sigma^9(z)}=3\wp(z)\wp^{\prime}(z)^2-\frac14 \wp^{\prime \prime}(z)^2,
\end{equation}
and
\begin{align*}
\wp^{\prime}(z)^2=&4\wp(z)^3-g_2\wp(z)-g_3,  \\
g_2(\omega_1, \omega_2)=&60
\sum_
{
\tiny{
\begin{matrix}
(m,n)\in\mathbb{Z}^2 \\ 
(m,n)\neq(0,0)
\end{matrix}
}
}
\frac{1}{(m\omega_1+n\omega_2)^4},  \,\,
g_3(\omega_1, \omega_2)=140
\sum_
{
\tiny{
\begin{matrix}
(m,n)\in\mathbb{Z}^2 \\ 
(m,n)\neq(0,0)
\end{matrix}
}
}
\frac{1}{(m\omega_1+n\omega_2)^6}, 
\end{align*}
which implies 
\begin{equation}
\wp^{\prime \prime}(z)=6\wp^2(z)-\frac12g_2, \,\,
\wp^{\prime \prime \prime}(z)=12\wp(z) \wp^{\prime}(z).
\end{equation}
Moreover, from Farkas and Kra \cite[pp. 124]{Farkas-Kra}, we recall 
\begin{align*}
\wp(z;1, \tau)=&
\frac{1}{3}
\frac
{
\theta^{\prime  \prime \prime} 
\left[
\begin{array}{c}
1 \\
1
\end{array}
\right]
}
{
\theta^{\prime } 
\left[
\begin{array}{c}
1 \\
1
\end{array}
\right]
}-
\frac{d^2}{dz^2} \log 
\theta
\left[
\begin{array}{c}
1 \\
1
\end{array}
\right](z,\tau),  \\
\sigma(z;\omega_1, \omega_2)=&
\exp
\left(
\frac{\eta_1 z^2}{2\omega_1}
\right)
\frac{\omega_1}
{ 
\theta^{\prime}
\left[
\begin{array}{c}
1 \\
1
\end{array}
\right] }
\theta
\left[
\begin{array}{c}
1 \\
1
\end{array}
\right]
\left(
\frac{z}{\omega_1},
\tau
\right),
\end{align*}
where $\omega_2/\omega_1=\tau\in\mathbb{H}^2.$ 
\par
Therefore, we obtain the following theta functional formulas:

\begin{theorem}
\label{theta-derivative }
{\it
For every $z\in\mathbb{C},$ we have
\begin{equation}
\label{eqn:theta-derivative-3rd}
\frac{d^3}{dz^3}\log 
\theta
\left[
\begin{array}{c}
1\\
1
\end{array}
\right](z)
=
\theta^{\prime}
\left[
\begin{array}{c}
1\\
1
\end{array}
\right]^3
\frac
{
\theta
\left[
\begin{array}{c}
1\\
1
\end{array}
\right](2z)
}
{
\theta^4
\left[
\begin{array}{c}
1\\
1
\end{array}
\right](z)
},
\end{equation}
\begin{align}
\frac
{
\theta^{\prime}
\left[
\begin{array}{c}
1\\
1
\end{array}
\right]^8
\theta
\left[
\begin{array}{c}
1\\
1
\end{array}
\right](3z)
}
{
\theta^9
\left[
\begin{array}{c}
1\\
1
\end{array}
\right](z)
}
=&
3
\left\{
\frac{1}{3}
\frac
{
\theta^{\prime  \prime \prime} 
\left[
\begin{array}{c}
1 \\
1
\end{array}
\right]
}
{
\theta^{\prime } 
\left[
\begin{array}{c}
1 \\
1
\end{array}
\right]
}-
\frac{d^2}{dz^2} \log 
\theta
\left[
\begin{array}{c}
1 \\
1
\end{array}
\right](z)
\right\}
\left\{
\frac
{
\theta^{\prime}
\left[
\begin{array}{c}
1\\
1
\end{array}
\right]^3
\theta
\left[
\begin{array}{c}
1\\
1
\end{array}
\right](2z)
}
{
\theta^4
\left[
\begin{array}{c}
1\\
1
\end{array}
\right](z)
}
\right\}^2    \notag    \\
&\quad
-\frac14
\left\{
\frac{d^4}{dz^4} \log 
\theta
\left[
\begin{array}{c}
1 \\
1
\end{array}
\right](z)
\right\}^2
\label{eqn:theta-derivative-4th}
\end{align}
and
\begin{equation}
\label{eqn:theta-derivative-5th}
\frac{d^5}{dz^5}\log 
\theta
\left[
\begin{array}{c}
1\\
1
\end{array}
\right](z)
=
12
\left\{
\frac{1}{3}
\frac
{
\theta^{\prime  \prime \prime} 
\left[
\begin{array}{c}
1 \\
1
\end{array}
\right]
}
{
\theta^{\prime } 
\left[
\begin{array}{c}
1 \\
1
\end{array}
\right]
}-
\frac{d^2}{dz^2} \log 
\theta
\left[
\begin{array}{c}
1 \\
1
\end{array}
\right](z)
\right\}
\frac
{
\theta^{\prime}
\left[
\begin{array}{c}
1\\
1
\end{array}
\right]^3
\theta
\left[
\begin{array}{c}
1\\
1
\end{array}
\right](2z)
}
{
\theta^4
\left[
\begin{array}{c}
1\\
1
\end{array}
\right](z)
}. 
\end{equation}
}
\end{theorem}

\begin{corollary}
\label{coro:pro-series-3order}
{\it
For every $\tau\in\mathbb{H}^2,$ we have 
\begin{equation}
\frac{\eta^9(\tau)}{\eta^3(3\tau)}
=
1-9\sum_{n=1}^{\infty}
q^n \left( \sum_{d|n}d^2 \left( \frac{d}{3} \right)  \right)
\end{equation}
and
\begin{equation}
\frac{\eta^9(3\tau)}{\eta^3(\tau)}
=
\sum_{n=1}^{\infty}
q^n \left(  \sum_{d|n}d^2 \left( \frac{n/d}{3} \right)  \right),
\end{equation}
where $q=\exp(2\pi i \tau).$ 
}
\end{corollary}

\begin{proof}
The corollary can be proved by substituting 
$z=-1/3$ or $z=-\tau/3$ in equation (\ref{eqn:theta-derivative-3rd}) 
and applying Jacobi's triple product identity (\ref{eqn:Jacobi-triple}). 
\end{proof}

\section{Preliminary results}
\label{sec:proof-preliminary}

\begin{proposition}
\label{prop:1st-derivative}
{\it
For every $\tau \in\mathbb{H}^2,$ we have 
\begin{equation}
\frac{
\theta^{\prime}
\left[
\begin{array}{c}
1 \\
\frac13
\end{array}
\right]  
}
{ 
\theta
\left[
\begin{array}{c}
1 \\
\frac13
\end{array}
\right]  
}=
-\frac{\pi}{\sqrt{3}} a(\tau), \,\,\mathit{and} \,\,
\frac{
\theta^{\prime}
\left[
\begin{array}{c}
\frac13 \\
1
\end{array}
\right]  
}
{ 
\theta
\left[
\begin{array}{c}
\frac13 \\
1
\end{array}
\right]  
}=
\frac{\pi i}{3} a(\tau/3). 
\end{equation}
}
\end{proposition}

\begin{proof}
The proposition follows from Jacobi's triple product identity (\ref{eqn:Jacobi-triple}). 
\end{proof}

\begin{proposition}
\label{application-residue-theorem}
{\it
For every $\tau \in\mathbb{H}^2,$ we have 
\begin{equation}
\label{eqn:relation-for-thm-1-1/3}
3
\frac{
\theta^{\prime\prime}
\left[
\begin{array}{c}
1 \\
\frac13
\end{array}
\right]
}
{
\theta
\left[
\begin{array}{c}
1 \\
\frac13
\end{array}
\right]
}
-
\frac
{
\theta^{\prime\prime \prime}
\left[
\begin{array}{c}
1 \\
1
\end{array}
\right]
}
{
\theta^{\prime}
\left[
\begin{array}{c}
1 \\
1
\end{array}
\right]
}
+
6
\left\{
\frac
{
\theta^{\prime} 
\left[
\begin{array}{c}
1 \\
\frac13
\end{array}
\right]
}
{
\theta
\left[
\begin{array}{c}
1 \\
\frac13
\end{array}
\right]
}
\right\}^2=0,
\end{equation}
 \begin{equation}
\label{eqn:relation-for-thm-1/3-1}
3
\frac{
\theta^{\prime\prime}
\left[
\begin{array}{c}
\frac13 \\
1
\end{array}
\right]
}
{
\theta
\left[
\begin{array}{c}
\frac13 \\
1
\end{array}
\right]
}
-
\frac
{
\theta^{\prime\prime \prime}
\left[
\begin{array}{c}
1 \\
1
\end{array}
\right]
}
{
\theta^{\prime}
\left[
\begin{array}{c}
1 \\
1
\end{array}
\right]
}
+
6
\left\{
\frac
{
\theta^{\prime} 
\left[
\begin{array}{c}
\frac13 \\
1
\end{array}
\right]
}
{
\theta
\left[
\begin{array}{c}
\frac13 \\
1
\end{array}
\right]
}
\right\}^2=0, 
\end{equation}
 \begin{equation}
\label{eqn:relation-for-thm-1/3-1/3}
3
\frac{
\theta^{\prime\prime}
\left[
\begin{array}{c}
\frac13 \\
\frac13
\end{array}
\right]
}
{
\theta
\left[
\begin{array}{c}
\frac13 \\
\frac13
\end{array}
\right]
}
-
\frac
{
\theta^{\prime\prime \prime}
\left[
\begin{array}{c}
1 \\
1
\end{array}
\right]
}
{
\theta^{\prime}
\left[
\begin{array}{c}
1 \\
1
\end{array}
\right]
}
+
6
\left\{
\frac
{
\theta^{\prime} 
\left[
\begin{array}{c}
\frac13 \\
\frac13
\end{array}
\right]
}
{
\theta
\left[
\begin{array}{c}
\frac13 \\
\frac13
\end{array}
\right]
}
\right\}^2=0,
\end{equation}
and
 \begin{equation}
\label{eqn:relation-for-thm-1/3-5/3}
3
\frac{
\theta^{\prime\prime}
\left[
\begin{array}{c}
\frac13 \\
\frac53
\end{array}
\right]
}
{
\theta
\left[
\begin{array}{c}
\frac13 \\
\frac53
\end{array}
\right]
}
-
\frac
{
\theta^{\prime\prime \prime}
\left[
\begin{array}{c}
1 \\
1
\end{array}
\right]
}
{
\theta^{\prime}
\left[
\begin{array}{c}
1 \\
1
\end{array}
\right]
}
+
6
\left\{
\frac
{
\theta^{\prime} 
\left[
\begin{array}{c}
\frac13 \\
\frac53
\end{array}
\right]
}
{
\theta
\left[
\begin{array}{c}
\frac13 \\
\frac53
\end{array}
\right]
}
\right\}^2=0.
\end{equation}
}
\end{proposition}

\begin{proof}
Consider the following elliptic functions:
\begin{equation*}
\varphi_1(z)
=
\frac
{
\theta^3
\left[
\begin{array}{c}
1 \\
\frac13
\end{array}
\right](z) 
}
{
\theta^3
\left[
\begin{array}{c}
1 \\
1
\end{array}
\right](z)
}, \quad 
\varphi_2(z)
=
\frac
{
\theta^3
\left[
\begin{array}{c}
\frac13 \\
1
\end{array}
\right](z) 
}
{
\theta^3
\left[
\begin{array}{c}
1 \\
1
\end{array}
\right](z)
}, 
\quad 
\varphi_3(z)
=
\frac
{
\theta^3
\left[
\begin{array}{c}
\frac13 \\
\frac13
\end{array}
\right](z) 
}
{
\theta^3
\left[
\begin{array}{c}
1 \\
1
\end{array}
\right](z)
}, 
\quad 
\varphi_4(z)
=
\frac
{
\theta^3
\left[
\begin{array}{c}
\frac13 \\
\frac53
\end{array}
\right](z) 
}
{
\theta^3
\left[
\begin{array}{c}
1 \\
1
\end{array}
\right](z)
}. 
\end{equation*}
In the fundamental parallelogram, 
the pole of $\varphi_j(z) \,\,(j=1,2,3,4)$ is $z=0,$ 
which implies that 
$\mathrm{Res} (\varphi_j(z), 0)=0.$ 
Therefore, the proposition follows. 
\end{proof}

\begin{proposition}
\label{prop:E2}
{\it
For every $\tau\in \mathbb{H}^2,$ we have 
\begin{equation*}
\frac
{
\theta^{\prime\prime \prime}
\left[
\begin{array}{c}
1 \\
1
\end{array}
\right]
}
{
\theta^{\prime}
\left[
\begin{array}{c}
1 \\
1
\end{array}
\right]
}
=4 \pi i \frac{d}{d\tau} \log \theta^{\prime}
\left[
\begin{array}{c}
1 \\
1
\end{array}
\right]
=
-\pi^2E_2(q), \,\,q=\exp(2\pi i \tau). 
\end{equation*}
}
\end{proposition}

\begin{proof}
The proposition follows from Jacobi's triple product identity (\ref{eqn:Jacobi-triple}). 
\end{proof}

\section{Proof of Theorem \ref{thm-level3-E_2(q)}}
\label{sec:proof:level3-E_2(q)}

\subsection{Notations}
For every $\tau\in\mathbb{H}^2,$ set $q=\exp(2\pi i \tau)$ and 
\begin{equation*}
X=
\frac
{
\theta^{\prime}
\left[
\begin{array}{c}
1 \\
\frac13
\end{array}
\right]
}
{
\theta
\left[
\begin{array}{c}
1 \\
\frac13
\end{array}
\right]
},  \,\,
Y=
\frac
{
\theta^{\prime  \prime  \prime}
\left[
\begin{array}{c}
1 \\
1
\end{array}
\right]
}
{
\theta^{\prime  }
\left[
\begin{array}{c}
1 \\
1
\end{array}
\right]
}, \,\,
Z=
\frac
{
\theta^{\prime  }
\left[
\begin{array}{c}
1 \\
1
\end{array}
\right]^3
}
{
\theta
\left[
\begin{array}{c}
1 \\
\frac13
\end{array}
\right]^3
}. 
\end{equation*}

\subsection{The ODE for $P(q)$}

\begin{proof}
From equation (\ref{eqn:relation-for-thm-1-1/3}), we have 
\begin{equation}
\frac{
\theta^{\prime \prime}
\left[
\begin{array}{c}
1 \\
\frac13
\end{array}
\right]  
}
{ 
\theta
\left[
\begin{array}{c}
1 \\
\frac13
\end{array}
\right]
}
=\frac13 Y-2X^2.
\end{equation}
By substituting $z=-1/3$ in equation (\ref{eqn:theta-derivative-3rd}), we obtain
\begin{equation}
\frac{
\theta^{\prime \prime \prime}
\left[
\begin{array}{c}
1 \\
\frac13
\end{array}
\right]  
}
{ 
\theta
\left[
\begin{array}{c}
1 \\
\frac13
\end{array}
\right]
}
=
-8 X^3+XY+Z. 
\end{equation}
\par
Since 
\begin{align*}
\frac{
\theta^{\prime \prime \prime}
\left[
\begin{array}{c}
1 \\
\frac13
\end{array}
\right]  
}
{ 
\theta
\left[
\begin{array}{c}
1 \\
\frac13
\end{array}
\right]
}
=&
4 \pi i
\frac{d}{d\tau}
\left\{
\frac{
\theta^{\prime }
\left[
\begin{array}{c}
1 \\
\frac13
\end{array}
\right]  
}
{ 
\theta
\left[
\begin{array}{c}
1 \\
\frac13
\end{array}
\right]
}
\right\}
+
\frac{
\theta^{\prime }
\left[
\begin{array}{c}
1 \\
\frac13
\end{array}
\right]  
}
{ 
\theta
\left[
\begin{array}{c}
1 \\
\frac13
\end{array}
\right]
}
\cdot
\frac{
\theta^{\prime \prime }
\left[
\begin{array}{c}
1 \\
\frac13
\end{array}
\right]  
}
{ 
\theta
\left[
\begin{array}{c}
1 \\
\frac13
\end{array}
\right]
},  \\
=&4\pi i X^{\prime}+\frac13 XY -2X^3, 
\end{align*}
it follows that 
\begin{equation}
4\pi i X^{\prime}=-6X^3+\frac23 XY+Z,  
\end{equation}
where ${ }^{\prime}=d/d\tau. $ 
Considering $d/d\tau=2 \pi i q d/dq,$ we obtain the ODE for $P(q).$ 
\end{proof}

\subsection{The ODE for $Q(q)$}

\begin{proof}
By substituting $z=-1/3$ in equation (\ref{eqn:theta-derivative-4th}), we have
\begin{equation}
\pm 6XZ=
\frac{
\theta^{(4)}
\left[
\begin{array}{c}
1 \\
\frac13
\end{array}
\right]  
}
{ 
\theta
\left[
\begin{array}{c}
1 \\
\frac13
\end{array}
\right]
}
-10X^4+4X^2Y-4XZ-\frac13Y^2. 
\end{equation}
Comparing the coefficients of the $q$-series, we obtain 
\begin{equation}
- 6XZ=
\frac{
\theta^{(4)}
\left[
\begin{array}{c}
1 \\
\frac13
\end{array}
\right]  
}
{ 
\theta
\left[
\begin{array}{c}
1 \\
\frac13
\end{array}
\right]
}
-10X^4+4X^2Y-4XZ-\frac13Y^2,
\end{equation}
which implies that 
\begin{equation}
\frac{
\theta^{(4)}
\left[
\begin{array}{c}
1 \\
\frac13
\end{array}
\right]  
}
{ 
\theta
\left[
\begin{array}{c}
1 \\
\frac13
\end{array}
\right]
}
=
10X^4-4X^2Y-2XZ+\frac13Y^2.
\end{equation}
\par
Since 
\begin{align*}
\frac{
\theta^{(4)}
\left[
\begin{array}{c}
1 \\
\frac13
\end{array}
\right]  
}
{ 
\theta
\left[
\begin{array}{c}
1 \\
\frac13
\end{array}
\right]
}
=&
4 \pi i
\frac{d}{d\tau}
\left\{
\frac{
\theta^{\prime \prime}
\left[
\begin{array}{c}
1 \\
\frac13
\end{array}
\right]  
}
{ 
\theta
\left[
\begin{array}{c}
1 \\
\frac13
\end{array}
\right]
}
\right\}
+
\left\{
\frac{
\theta^{\prime \prime }
\left[
\begin{array}{c}
1 \\
\frac13
\end{array}
\right]  
}
{ 
\theta
\left[
\begin{array}{c}
1 \\
\frac13
\end{array}
\right]
}
\right\}^2,   \\
=&
\frac43 \pi i Y^{\prime}+28X^4-4X^2Y-4XZ+\frac19Y^2, 
\end{align*}
it follows that 
\begin{equation}
4\pi i Y^{\prime}=-54X^4+6XZ+\frac23 Y^2,   
\end{equation}
where ${ }^{\prime}=d/d\tau. $ 
Considering $d/d\tau=2 \pi i q d/dq,$ we obtain the ODE for $Q(q).$ 
\end{proof}

\subsection{The ODE for $R(q)$}

\begin{proof}
By substituting $z=-1/3$ in equation (\ref{eqn:theta-derivative-5th}), we have
\begin{equation}
\frac{
\theta^{(5)}
\left[
\begin{array}{c}
1 \\
\frac13
\end{array}
\right]  
}
{ 
\theta
\left[
\begin{array}{c}
1 \\
\frac13
\end{array}
\right]
}
=
106X^5-\frac{80}{3}X^3 Y-14X^2 Z+\frac53XY^2+\frac{10}{3}YZ,
\end{equation}
and
\begin{equation}
\label{eqn:a^2(q)b^3(q)}
a^2(q)b^3(q)=
a^2(q) \frac{(q;q)_{\infty}^9}{(q^3;q^3)_{\infty}^3}
=
1+3\sum_{n=1}^{\infty} q^n
\left(
\sum_{d|n} d^4
\left(
\frac{d}{3}
\right)
\right). 
\end{equation}
\par
Since 
\begin{align*}
\frac{
\theta^{(5)}
\left[
\begin{array}{c}
1 \\
\frac13
\end{array}
\right]  
}
{ 
\theta
\left[
\begin{array}{c}
1 \\
\frac13
\end{array}
\right]
}
=&
4 \pi i
\frac{d}{d\tau}
\left\{
\frac{
\theta^{(3) }
\left[
\begin{array}{c}
1 \\
\frac13
\end{array}
\right]  
}
{ 
\theta
\left[
\begin{array}{c}
1 \\
\frac13
\end{array}
\right]
}
\right\}
+
\frac{
\theta^{(3) }
\left[
\begin{array}{c}
1 \\
\frac13
\end{array}
\right]  
}
{ 
\theta
\left[
\begin{array}{c}
1 \\
\frac13
\end{array}
\right]
}
\cdot
\frac{
\theta^{\prime \prime }
\left[
\begin{array}{c}
1 \\
\frac13
\end{array}
\right]  
}
{ 
\theta
\left[
\begin{array}{c}
1 \\
\frac13
\end{array}
\right]
},  \\
=&
4\pi i Z^{\prime}+106X^5-\frac{80}{3}X^3Y-20X^2Z+\frac53XY^2+\frac43YZ, 
\end{align*}
it follows that 
\begin{equation}
4\pi i Z^{\prime}=6X^2Z+2YZ, 
\end{equation}
where ${ }^{\prime}=d/d\tau. $ 
Considering $d/d\tau=2 \pi i q d/dq,$ we obtain the ODE for $R(q).$ 
\end{proof}

\subsection{Note on $E_4$ and $E_6$}

\begin{theorem}
\label{thm-E4-E6-level3-q}
{\it
For $q\in\mathbb{C}$ with $|q|<1,$ we have 
\begin{equation}
\label{eqn-E4-E6-level3}
E_4(q)=9a^4(q)-8a(q)b^3(q), \,\,
E_6(q)=-27a^6(q)+36a^3(q)b^3(q)-8b^6(q). 
\end{equation}
}
\end{theorem}

\begin{proof}
For the proof, we use 
\begin{equation*}
\wp=\wp(z;1,\tau)
=
\frac13
\frac
{
\theta^{\prime  \prime  \prime}
\left[
\begin{array}{c}
1 \\
1
\end{array}
\right]
}
{
\theta^{\prime  }
\left[
\begin{array}{c}
1 \\
1
\end{array}
\right]
}
-
\frac{d^2}{dz^2}
\log
\theta
\left[
\begin{array}{c}
1 \\
1
\end{array}
\right](z). 
\end{equation*} 
\par
We first note that 
\begin{equation*}
\wp^{\prime \prime}=6\wp^2-g_2, \,\,
g_2(1,\tau)=\frac{4\pi^4}{3}E_4(q).
\end{equation*}
Substituting $z=-1/3,$ we have 
\begin{equation}
\label{eqn:g2-X-Z(1)}
g_2=108X^4-12XZ,
\end{equation}
which implies 
\begin{equation}
E_4(q)=9P^4-8PR. 
\end{equation}
\par
We next note that 
\begin{equation*}
(\wp^{\prime})^2=4\wp^3-g_2\wp-g_3, \,\,
g_3(1,\tau)=\frac{8\pi^6}{27}E_6(q).
\end{equation*}
Substituting $z=-1/3,$ we obtain 
\begin{equation}
\label{eqn:g3-X-Z(1)}
g_3=-216X^6+36X^3Z-Z^2,
\end{equation}
which implies 
\begin{equation*}
E_6(q)=-27P^6+36P^3R-8R^2. 
\end{equation*}
\end{proof}

\begin{theorem}
\label{thm:J-a-b^3}
{\it
For every $\tau\in\mathbb{H}^2,$ we have 
\begin{align*}
J(\tau)=&J(q)
=\frac{12^3 g_2^3}{g_2^3-27 g_3^2}  \\
=&\frac{27 a^3(q)\left(9 a^3(q)  -8  b^3(q)\right)^3}{b^9(q) \left(a^3(q)  -b^3(q)\right)}
=\frac{27a^3(q)(a^3(q)+8c^3(q))^3  }{b^9(q)c^3(q)  }, \,\,q=\exp(2\pi i \tau).
\end{align*}
}
\end{theorem}

\begin{proof}
By equations (\ref{eqn:g2-X-Z(1)}) and (\ref{eqn:g3-X-Z(1)}), we have 
\begin{equation*}
J(\tau)=
\frac{110592 \left(9 X^4-X Z\right)^3}{Z^3 \left(8 X^3-Z\right)},
\end{equation*}
which proves the theorem. 
The second formula follows from Ramanujan's cubic identity (\ref{eqn:Ramanujan-cubic-intro}). 
\end{proof}

\subsection{Proof of Ramanujan's system of ODEs (\ref{eqn:Ramanujan-ODE})}

\begin{theorem}
{\it
For $q\in\mathbb{C}$ with $|q|<1,$ 
we have 
\begin{equation*}
q\frac{dE_2}{dq}=\frac{(E_2)^2-E_4  }{12  }, \,\,
q\frac{dE_4}{dq}=\frac{E_2 E_4-E_6  }{3  }, \,\,
q\frac{dE_6}{dq}=\frac{E_2 E_6-(E_4)^2  }{2  }. 
\end{equation*}
}
\end{theorem}

\begin{proof}
By Theorems \ref{thm-level3-E_2(q)} and \ref{thm-E4-E6-level3-q}, 
we first note that 
\begin{equation*}
q\frac{d}{dq}E_2=q\frac{d}{dq}Q=\frac{Q^2-(9P^2-8PR)}{12}
=\frac{E_2^2-E_4}{12}. 
\end{equation*}
We next see that 
\begin{align*}
q\frac{d}{dq}E_4=&
q\frac{d}{dq}(9P^4-8PR)
=9P^6+3P^4Q-12P^3R+\frac83R^2-\frac83PQR  \\
=&\frac{E_2E_4-E_6}{3}. 
\end{align*}
Finally, we obtain 
\begin{align*}
q\frac{d}{dq}E_6=&
q\frac{d}{dq}(-27P^6+36P^3R-8R^2)\\  
=&
-\frac{81}{2}P^8-\frac{27}{2}P^6Q+72P^5R+18P^3QR-32P^2R^2-4QR^2 \\
=&\frac{E_2E_6-E_4^2}{2}. 
\end{align*}
\end{proof}

\section{Proof of Theorem \ref{thm-level3-E_2(q^3)}}
\label{sec:proof:level3-E_2(q^3)}

\subsection{Notations}
For every $\tau\in\mathbb{H}^2,$ set $y=\exp(2\pi i \tau/3), \,q=y^3$ and 
\begin{equation*}
X=
\frac
{
\theta^{\prime}
\left[
\begin{array}{c}
\frac13 \\
1
\end{array}
\right]
}
{
\theta
\left[
\begin{array}{c}
\frac13 \\
1
\end{array}
\right]
},  \,\,
Y=
\frac
{
\theta^{\prime  \prime  \prime}
\left[
\begin{array}{c}
1 \\
1
\end{array}
\right]
}
{
\theta^{\prime  }
\left[
\begin{array}{c}
1 \\
1
\end{array}
\right]
}, \,\,
Z=
-
\frac
{
\theta^{\prime  }
\left[
\begin{array}{c}
1 \\
1
\end{array}
\right]^3
}
{
\theta
\left[
\begin{array}{c}
\frac13 \\
1
\end{array}
\right]^3
}. 
\end{equation*}

\subsection{The ODE for $P(q)$}

\begin{proof}
From equation (\ref{eqn:relation-for-thm-1/3-1}), we have 
\begin{equation}
\frac{
\theta^{\prime \prime}
\left[
\begin{array}{c}
\frac13 \\
1
\end{array}
\right]  
}
{ 
\theta
\left[
\begin{array}{c}
\frac13 \\
1
\end{array}
\right]
}
=\frac13 Y-2X^2.
\end{equation}
By substituting $z=-\tau/3$ in equation (\ref{eqn:theta-derivative-3rd}), we obtain
\begin{equation}
\frac{
\theta^{\prime \prime \prime}
\left[
\begin{array}{c}
\frac13 \\
1
\end{array}
\right]  
}
{ 
\theta
\left[
\begin{array}{c}
\frac13 \\
1
\end{array}
\right]
}
=
-8 X^3+XY+Z. 
\end{equation}
\par
Since 
\begin{align*}
\frac{
\theta^{\prime \prime \prime}
\left[
\begin{array}{c}
\frac13 \\
1
\end{array}
\right]  
}
{ 
\theta
\left[
\begin{array}{c}
\frac13 \\
1
\end{array}
\right]
}
=&
4 \pi i
\frac{d}{d\tau}
\left\{
\frac{
\theta^{\prime }
\left[
\begin{array}{c}
\frac13 \\
1
\end{array}
\right]  
}
{ 
\theta
\left[
\begin{array}{c}
\frac13 \\
1
\end{array}
\right]
}
\right\}
+
\frac{
\theta^{\prime }
\left[
\begin{array}{c}
\frac13 \\
1
\end{array}
\right]  
}
{ 
\theta
\left[
\begin{array}{c}
\frac13 \\
1
\end{array}
\right]
}
\cdot
\frac{
\theta^{\prime \prime }
\left[
\begin{array}{c}
\frac13 \\
1
\end{array}
\right]  
}
{ 
\theta
\left[
\begin{array}{c}
\frac13 \\
1
\end{array}
\right]
},  \\
=&4\pi i X^{\prime}+\frac13 XY -2X^3, 
\end{align*}
it follows that 
\begin{equation}
4\pi i X^{\prime}=-6X^3+\frac23 XY+Z,  
\end{equation}
where ${ }^{\prime}=d/d\tau. $ 
Changing $\tau\rightarrow 3\tau,$ we obtain the ODE for $P(q).$ 
\end{proof}

\subsection{The ODE for $Q(q)$}

\begin{proof}
By substituting $z=-\tau/3$ in equation (\ref{eqn:theta-derivative-4th}), we have
\begin{equation}
\pm 6XZ=
\frac{
\theta^{(4)}
\left[
\begin{array}{c}
\frac13 \\
1
\end{array}
\right]  
}
{ 
\theta
\left[
\begin{array}{c}
\frac13 \\
1
\end{array}
\right]
}
-10X^4+4X^2Y-4XZ-\frac13Y^2. 
\end{equation}
Comparing the coefficients of the $y$-series, we obtain 
\begin{equation}
- 6XZ=
\frac{
\theta^{(4)}
\left[
\begin{array}{c}
\frac13 \\
1
\end{array}
\right]  
}
{ 
\theta
\left[
\begin{array}{c}
\frac13 \\
1
\end{array}
\right]
}
-10X^4+4X^2Y-4XZ-\frac13Y^2,
\end{equation}
which implies that 
\begin{equation}
\frac{
\theta^{(4)}
\left[
\begin{array}{c}
\frac13 \\
1
\end{array}
\right]  
}
{ 
\theta
\left[
\begin{array}{c}
\frac13 \\
1
\end{array}
\right]
}
=
10X^4-4X^2Y-2XZ+\frac13Y^2.
\end{equation}
\par
Since 
\begin{align*}
\frac{
\theta^{(4)}
\left[
\begin{array}{c}
\frac13 \\
1
\end{array}
\right]  
}
{ 
\theta
\left[
\begin{array}{c}
\frac13 \\
1
\end{array}
\right]
}
=&
4 \pi i
\frac{d}{d\tau}
\left\{
\frac{
\theta^{\prime \prime}
\left[
\begin{array}{c}
\frac13 \\
1
\end{array}
\right]  
}
{ 
\theta
\left[
\begin{array}{c}
\frac13 \\
1
\end{array}
\right]
}
\right\}
+
\left\{
\frac{
\theta^{\prime \prime }
\left[
\begin{array}{c}
\frac13 \\
1
\end{array}
\right]  
}
{ 
\theta
\left[
\begin{array}{c}
\frac13 \\
1
\end{array}
\right]
}
\right\}^2,   \\
=&
\frac43 \pi i Y^{\prime}+28X^4-4X^2Y-4XZ+\frac19Y^2, 
\end{align*}
it follows that 
\begin{equation}
4\pi i Y^{\prime}=-54X^4+6XZ+\frac23 Y^2,   
\end{equation}
where ${ }^{\prime}=d/d\tau. $ 
Changing $\tau\rightarrow 3\tau,$ we obtain the ODE for $Q(q).$ 
\end{proof}

\subsection{The ODE for $R(q)$}

\begin{proof}
By substituting $z=-\tau/3$ in equation (\ref{eqn:theta-derivative-5th}), we have
\begin{equation}
\frac{
\theta^{(5)}
\left[
\begin{array}{c}
\frac13 \\
1
\end{array}
\right]  
}
{ 
\theta
\left[
\begin{array}{c}
\frac13 \\
1
\end{array}
\right]
}
=
106X^5-\frac{80}{3}X^3 Y-14X^2 Z+\frac53XY^2+\frac{10}{3}YZ,
\end{equation}
and
\begin{equation}
\label{eqn:a^2(q)c^3(q)}
a^2(q) c^3(q)=
a^2(q) \cdot
27
q
\frac{(q^3;q^3)_{\infty}^9  }{(q;q)_{\infty}^3  }
=
27
\sum_{n=1}^{\infty}
q^n 
\left(
\sum_{d|n}
d^4
\left(
\frac{n/d}{3}
\right)
\right). 
\end{equation}
\par
Since 
\begin{align*}
\frac{
\theta^{(5)}
\left[
\begin{array}{c}
\frac13 \\
1
\end{array}
\right]  
}
{ 
\theta
\left[
\begin{array}{c}
\frac13 \\
1
\end{array}
\right]
}
=&
4 \pi i
\frac{d}{d\tau}
\left\{
\frac{
\theta^{(3) }
\left[
\begin{array}{c}
\frac13 \\
1
\end{array}
\right]  
}
{ 
\theta
\left[
\begin{array}{c}
\frac13 \\
1
\end{array}
\right]
}
\right\}
+
\frac{
\theta^{(3) }
\left[
\begin{array}{c}
\frac13 \\
1
\end{array}
\right]  
}
{ 
\theta
\left[
\begin{array}{c}
\frac13 \\
1
\end{array}
\right]
}
\cdot
\frac{
\theta^{\prime \prime }
\left[
\begin{array}{c}
\frac13 \\
1
\end{array}
\right]  
}
{ 
\theta
\left[
\begin{array}{c}
\frac13 \\
1
\end{array}
\right]
},  \\
=&
4\pi i Z^{\prime}+106X^5-\frac{80}{3}X^3Y-20X^2Z+\frac53XY^2+\frac43YZ, 
\end{align*}
it follows that 
\begin{equation}
4\pi i Z^{\prime}=6X^2Z+2YZ, 
\end{equation}
where ${ }^{\prime}=d/d\tau. $ 
Changing $\tau\rightarrow 3\tau,$ we obtain the ODE for $R(q).$ 
\end{proof}

\subsection{Note on $E_4$ and $E_6$}

\begin{theorem}
\label{thm-E4-E6-level3-q^3}
{\it
For $q\in\mathbb{C}$ with $|q|<1,$ 
we have 
\begin{equation}
\label{eqn-E4-E6-level3-q^3}
E_4(q^3)=a^4(q)-\frac89a(q)c^3(q), \,\,
E_6(q^3)=a^6(q)-\frac43a^3(q)c^3(q)+\frac{8}{27}c^6(q). 
\end{equation}
}
\end{theorem}

\begin{proof}
For the proof, we use 
\begin{equation*}
\wp=\wp(z;1,\tau)
=
\frac13
\frac
{
\theta^{\prime  \prime  \prime}
\left[
\begin{array}{c}
1 \\
1
\end{array}
\right]
}
{
\theta^{\prime  }
\left[
\begin{array}{c}
1 \\
1
\end{array}
\right]
}
-
\frac{d^2}{dz^2}
\log
\theta
\left[
\begin{array}{c}
1 \\
1
\end{array}
\right](z). 
\end{equation*} 
\par
We first note that 
\begin{equation*}
\wp^{\prime \prime}=6\wp^2-g_2, \,\,
g_2(1,\tau)=\frac{4\pi^4}{3}E_4(q). 
\end{equation*}
Substituting $z=-\tau/3,$ we have 
\begin{equation}
\label{eqn:g2-X-Z(2)}
g_2=108X^4-12XZ,
\end{equation}
which implies 
\begin{equation}
E_4(q)=9P(y)^4-8P(y)R(y). 
\end{equation}
\par
We next note that 
\begin{equation*}
(\wp^{\prime})^2=4\wp^3-g_2\wp-g_3, \,\,
g_3(1,\tau)=\frac{8\pi^6}{27}E_6(q). 
\end{equation*}
Substituting $z=-\tau/3,$ we obtain 
\begin{equation}
\label{eqn:g3-X-Z(2)}
g_3=-216X^6+36X^3Z-Z^2,
\end{equation}
which implies 
\begin{equation*}
E_6(q)=P(y)^6-36P(y)^3R(y)+216R(y)^2. 
\end{equation*}
Changing $\tau\rightarrow 3\tau,$ we obtain the theorem. 
\end{proof}

\begin{theorem}
\label{thm:J-a-c^3}
{\it
For every $\tau\in\mathbb{H}^2,$ we have 
\begin{equation*}
J(3\tau)=J(q^3)
=\frac{27 a^3(q) \left(9 a^3(q)  -8  c^3(q)\right)^3}{c^9(q) \left(a^3(q)  -c^3(q)\right)}
=
\frac
{
27 a^3(q)( a^3(q)+ 8 b^3(q))^3
}
{
b^3(q) c^9(q)
}, \,\,q=\exp(2\pi i \tau). 
\end{equation*}
}
\end{theorem}

\begin{proof}
By equations (\ref{eqn:g2-X-Z(2)}) and (\ref{eqn:g3-X-Z(2)}), we have 
\begin{equation*}
J(\tau)=
\frac
{
12^3
g_2^3
}
{
g_2^3-27 g_3^2
}
=
\frac{110592 \left(9 X^4-X Z\right)^3}{Z^3 \left(8 X^3-Z\right)},
\end{equation*}
which proves the theorem. 
The second formula follows from Ramanujan's cubic identity (\ref{eqn:Ramanujan-cubic-intro}). 
\end{proof}

\subsubsection*{Remark}
The formulas of Theorems \ref{thm-E4-E6-level3-q} and \ref{thm-E4-E6-level3-q^3} were proved in Cooper \cite[pp. 272]{Cooper}.

\section{Applications to number theory}
\label{sec:applications}

\begin{theorem}
(Farkas and Kra \cite[pp. 318]{Farkas-Kra})
\label{thm:1-1/3}
{\it
For every $\tau\in\mathbb{H}^2,$ 
we have 
\begin{equation*}
\frac{d}{d\tau} 
\log \frac{\eta(3\tau)}{\eta(\tau)}
+
\frac{1}{2\pi i } 
\left\{
\frac
{
\theta^{\prime} 
\left[
\begin{array}{c}
1 \\
\frac13
\end{array}
\right](0,\tau) 
}
{
\theta
\left[
\begin{array}{c}
1 \\
\frac13
\end{array}
\right](0,\tau)
}
\right\}^2=0. 
\end{equation*}
}
\end{theorem}

\begin{proof}
The heat equation (\ref{eqn:heat}) and equation (\ref{eqn:relation-for-thm-1-1/3}) implies that 
\begin{equation*}
4\pi i 
\frac{d}{d\tau}
\log 
\frac
{
\theta^3
\left[
\begin{array}{c}
1 \\
\frac13
\end{array}
\right]
}
{
\theta^{\prime}
\left[
\begin{array}{c}
1 \\
1
\end{array}
\right]
}
+
6
\left\{
\frac
{
\theta^{\prime} 
\left[
\begin{array}{c}
1 \\
\frac13
\end{array}
\right]
}
{
\theta
\left[
\begin{array}{c}
1 \\
\frac13
\end{array}
\right]
}
\right\}^2=0. 
\end{equation*}
The theorem follows from Jacobi's triple product identity (\ref{eqn:Jacobi-triple}). 
\end{proof}

\begin{theorem}
\label{thm:a^2(q)}
{\it
For $q\in\mathbb{C}$ with $|q|<1,$ we have 
\begin{equation*}
a^2(q)=1+12\sum_{n=1}^{\infty}(\sigma_1(n)-3\sigma_1(n/3))q^n
=\frac12\left\{   
-E_2(q)+3
E_2(q^3)
\right\}.
\end{equation*}
}
\end{theorem}

\begin{proof}
The theorem follows from Proposition \ref{prop:1st-derivative} and Theorem \ref{thm:1-1/3}. 
\end{proof}

\begin{theorem}
\label{thm:a^3(q)}
{\it
For $q\in\mathbb{C}$ with $|q|<1,$ we have 
\begin{equation*}
a^3(q)=1-9\sum_{n=1}^{\infty} q^n\left( \sum_{d|n} d^2\left( \frac{d}{3}  \right)  \right)
+27
\sum_{n=1}^{\infty} q^n\left( \sum_{d|n} d^2\left( \frac{n/d}{3}  \right)  \right). 
\end{equation*}
}
\end{theorem}

\begin{proof}
From Theorems \ref{thm-level3-E_2(q)} and \ref{thm-level3-E_2(q^3)}, we have 
\begin{equation}
\label{eqn:operator-normal}
\left(12q \frac{d}{dq}-E_2(q)  \right)  a(q)=3 a^3(q)-
4\frac
{
(q;q)_{\infty}^9
}
{
(q^3;q^3)_{\infty}^3
},
\end{equation}
and
\begin{equation}
\label{eqn:operator-not-normal}
\left(4q \frac{d}{dq}-E_2(q^3)  \right) a(q)=-a^3(q)+36q\frac{(q^3;q^3)_{\infty}^9}{(q;q)_{\infty}^3}. 
\end{equation}
Eliminating $q(d/dq)a(q),$ we obtain  
\begin{equation*}
2(1+12\sum_{n=1}^{\infty}(\sigma_1(n)-3\sigma_1(n/3)) q^n )a(q)
=6a^3(q)
-4
\frac{(q;q)_{\infty}^9}{(q^3;q^3)_{\infty}^3}
-108
q\frac{(q^3;q^3)_{\infty}^9}{(q;q)_{\infty}^3}.
\end{equation*}
Therefore, it follows from Theorem \ref{thm:a^2(q)} that
\begin{equation}
\label{eqn:a^3(q)}
a^3(q)=
\frac
{
(q;q)_{\infty}^9
}
{
(q^3;q^3)_{\infty}^3
}
+
27
q\frac{(q^3;q^3)_{\infty}^9}{(q;q)_{\infty}^3}
=
b^3(q)+c^3(q), 
\end{equation}
which proves the theorem. 
\end{proof}

\begin{theorem}
\label{thm:a^4(q)}
{\it
For $q\in\mathbb{C}$ with $|q|<1,$ we have 
\begin{equation*}
a^4(q)=1+24\sum_{n=1}^{\infty}(\sigma_3(n)+9\sigma_3(n/3))q^n. 
\end{equation*}
}
\end{theorem}

\begin{proof}
By Theorems \ref{thm-E4-E6-level3-q} and \ref{thm-E4-E6-level3-q^3}, 
we have 
\begin{equation*}
E_4(q)=9a^4(q)-8a(q)
\frac
{
(q;q)_{\infty}^9
}
{
(q^3;q^3)_{\infty}^3
},  \,\,
E_4(q^3)=a^4(q)-24a(q)
\cdot
q
\frac
{
(q^3;q^3)_{\infty}^9
}
{
(q;q)_{\infty}^3
}.
\end{equation*}
Considering equation (\ref{eqn:a^3(q)}), 
we obtain 
\begin{equation*}
10a^4(q)=E_4(q)+9E_4(q^3),
\end{equation*}
which proves the theorem. 
\end{proof}

\begin{theorem}
\label{thm:a^5(q)}
{\it
For $q\in\mathbb{C}$ with $|q|<1,$ we have 
\begin{equation*}
a^5(q)=
1+
3
\sum_{n=1}^{\infty}
q^n 
\left(
\sum_{d|n}
d^4
\left(
\frac{d}{3}
\right)
\right)
+
27
\sum_{n=1}^{\infty}
q^n 
\left(
\sum_{d|n}
d^4
\left(
\frac{n/d}{3}
\right)
\right). 
\end{equation*}
}
\end{theorem}

\begin{proof}
The theorem follows from equations (\ref{eqn:a^2(q)b^3(q)}), (\ref{eqn:a^2(q)c^3(q)}) and 
(\ref{eqn:a^3(q)}). 
\end{proof}

\begin{theorem}
\label{thm:a^6(q)}
{\it
For $q\in\mathbb{C}$ with $|q|<1,$ we have 
\begin{equation*}
a^6(q)=
1+
\frac{252}{13}
\sum_{n=1}^{\infty}
(\sigma_5(n)-27\sigma_5(n/3)) q^n
+
\frac{216}{13}
q(q;q)_{\infty}^6(q^3;q^3)_{\infty}^6. 
\end{equation*}
}
\end{theorem}

\begin{proof}
By Theorems \ref{thm-E4-E6-level3-q}, \ref{thm-E4-E6-level3-q^3} and equation (\ref{eqn:a^3(q)}), 
we have 
\begin{equation*}
E_6(q)-27E_6(q^3)
=
-
18 a^6(q)-8b^6(q)-8c^6(q),
\end{equation*}
which implies 
\begin{equation*}
9a^6(q)+4b^6(q)+4c^6(q)
=
13+
252
\sum_{n=1}^{\infty}
(\sigma_5(n)-27\sigma_5(n/3)) q^n.
\end{equation*}
By equation (\ref{eqn:a^3(q)}), we obtain 
\begin{equation*}
13a^6(q)
=
13+
252
\sum_{n=1}^{\infty}
(\sigma_5(n)-27\sigma_5(n/3)) q^n
+
8b^3(q) c^3(q),
\end{equation*}
which proves the theorem. 
\end{proof}

\subsubsection*{Remark}
Lomadze \cite{Lomadze} proved 
Theorems \ref{thm:a^2(q)}, \ref{thm:a^3(q)}, \ref{thm:a^4(q)}, and \ref{thm:a^5(q)} 
by means of specific Eisenstein series. 
He also treated 
\begin{equation*}
F_k=x_1^2+x_1x_2+x_2^2+\cdots+x_{2k-1}^2+x_{2k-1}x_{2k}+x_{2k}^2, \,\,(k=1,2,\ldots,17). 
\end{equation*}  
Based on Ramanujan's theory of theta functions, Cooper \cite[pp. 569]{Cooper} treats the case where $k=1,2,\ldots8, 10, 12.$ 
\par
For the elementary proof of Theorems \ref{thm:a^2(q)} and \ref{thm:a^4(q)}, 
readers are referred to the book by Williams \cite[pp. 224-227]{Williams}. 

\section{A selected example of Ramanujan's identity}
\label{sec:Ramanujan-a(q)}

\subsection{Farkas and Kra's cubic identity}

\begin{theorem}
(Farkas and Kra \cite[pp. 193]{Farkas-Kra})
\label{thm:FarkasKraCubic}
{\it
For every $\tau\in\mathbb{H}^2,$ we have 
\begin{equation}
\label{Farkas-Kra-cubic-(1)}
\theta^3
\left[
\begin{array}{c}
\frac13 \\
\frac13
\end{array}
\right]
+
\theta^3
\left[
\begin{array}{c}
\frac13 \\
\frac53
\end{array}
\right]
=
\theta^3
\left[
\begin{array}{c}
\frac13 \\
1
\end{array}
\right],
\end{equation}
and 
\begin{equation}
\label{Farkas-Kra-cubic-(2)}
\exp\left(\frac{\pi i}{3}\right)
\theta^3
\left[
\begin{array}{c}
\frac13 \\
\frac13
\end{array}
\right]
+
\exp\left(\frac{2\pi i}{3}\right)
\theta^3
\left[
\begin{array}{c}
\frac13 \\
\frac53
\end{array}
\right]
=
\theta^3
\left[
\begin{array}{c}
1 \\
\frac13
\end{array}
\right]. 
\end{equation}
}
\end{theorem}

\begin{proof}
Consider the following elliptic functions:
\begin{equation*}
\varphi(z)
=
\frac
{
\theta^3
\left[
\begin{array}{c}
1 \\
1
\end{array}
\right](z) 
}
{
\theta
\left[
\begin{array}{c}
\frac13 \\
\frac13
\end{array}
\right](z)
\theta
\left[
\begin{array}{c}
\frac13 \\
1
\end{array}
\right](z)
\theta
\left[
\begin{array}{c}
\frac13 \\
\frac53
\end{array}
\right](z)
}, 
\quad
\psi(z)
=
\frac
{
\theta^3
\left[
\begin{array}{c}
1 \\
1
\end{array}
\right](z) 
}
{
\theta
\left[
\begin{array}{c}
\frac13 \\
\frac13
\end{array}
\right](z)
\theta
\left[
\begin{array}{c}
1 \\
\frac13
\end{array}
\right](z)
\theta
\left[
\begin{array}{c}
\frac53 \\
\frac13
\end{array}
\right](z)
}. 
\end{equation*}
We use $\varphi(z)$ to prove equation (\ref{Farkas-Kra-cubic-(1)}). 
Equation (\ref{Farkas-Kra-cubic-(2)}) can be obtained by using $\psi(z)$ in the same way. 
\par
Note that 
in the fundamental parallelogram, 
the poles of $\varphi(z)$ are 
$(\tau+1)/3,$ $\tau/3,$ and $(\tau-1)/3.$ 
Direct computations yield 
\begin{equation*}
\mathrm{Res} \left(\varphi(z), \frac{\tau+1}{3}\right)=
-\frac
{
\theta^3
\left[
\begin{array}{c}
\frac13 \\
\frac13
\end{array}
\right]
}
{
\theta^{\prime}
\left[
\begin{array}{c}
1 \\
1
\end{array}
\right]
\theta^2
\left[
\begin{array}{c}
1 \\
\frac13
\end{array}
\right]
}, \quad
\mathrm{Res} \left(\varphi(z), \frac{\tau}{3}\right)=
\frac
{
\theta^3
\left[
\begin{array}{c}
\frac13 \\
1
\end{array}
\right]
}
{
\theta^{\prime}
\left[
\begin{array}{c}
1 \\
1
\end{array}
\right]
\theta^2
\left[
\begin{array}{c}
1 \\
\frac13
\end{array}
\right]
}, 
\end{equation*}
and
\begin{equation*}
\mathrm{Res} \left(\varphi(z), \frac{\tau-1}{3}\right)=
-\frac
{
\theta^3
\left[
\begin{array}{c}
\frac13 \\
\frac53
\end{array}
\right]
}
{
\theta^{\prime}
\left[
\begin{array}{c}
1 \\
1
\end{array}
\right]
\theta^2
\left[
\begin{array}{c}
1 \\
\frac13
\end{array}
\right]
}.
\end{equation*}
From the residue theorem, 
it follows that 
\begin{equation*}
\mathrm{Res} \left(\varphi(z), \frac{\tau+1}{3}\right)
+
\mathrm{Res} \left(\varphi(z), \frac{\tau}{3}\right)
+
\mathrm{Res} \left(\varphi(z), \frac{\tau-1}{3}\right)=0,
\end{equation*}
which implies equation (\ref{Farkas-Kra-cubic-(1)}). 
\end{proof}

\subsection{Ramanujan's identity}

\begin{theorem}
(Ramanujan \cite[pp. 346]{Berndt-Ramanujan})
\label{thm:ramanujan-level3}
{\it
For every $\tau\in \mathbb{H}^2,$ we have 
\begin{equation*}
a(q)=\frac
{
\eta^3(\tau/3)+3\eta^3(3\tau)
}
{
\eta(\tau)
}. 
\end{equation*}
}
\end{theorem}

\begin{proof}
From the results obtained by Farkas \cite{Farkas}, we recall the following identity:
\begin{align}
&
\frac{
6
\theta^{\prime} 
\left[
\begin{array}{c}
1 \\
\frac13
\end{array}
\right] 
(0,\tau)
}
{
\zeta_6
\theta^3
\left[
\begin{array}{c}
\frac13 \\
\frac13
\end{array}
\right] 
(0,\tau)
+
\theta^3
\left[
\begin{array}{c}
\frac13 \\
1
\end{array}
\right] 
(0,\tau)
+\zeta_6^5
\theta^3
\left[
\begin{array}{c}
\frac13 \\
\frac53
\end{array}
\right] 
(0,\tau)
}    \notag \\
=&
\frac{2\pi i q^{\frac{1}{12}}}
{
\displaystyle
\prod_{n=0}^{\infty} (1-q^{3n+1})(1-q^{3n+2})
}   
=
2\pi i \frac{e^{\frac{\pi i}{6}}}{\sqrt{3}}
\frac
{
\theta
\left[
\begin{array}{c}
1 \\
\frac13
\end{array}
\right] 
(0,\tau)
}
{
\theta
\left[
\begin{array}{c}
\frac13 \\
1
\end{array}
\right] 
(0,3\tau)
},  
\label{eqn:Farkas} 
\end{align}
where $q=\exp(2 \pi i \tau)$ and $\zeta_6=\exp(2 \pi i/6).$ 
\par
Theorem \ref{thm:FarkasKraCubic} yields that 
\begin{equation*}
\theta^3
\left[
\begin{array}{c}
\frac13 \\
\frac13
\end{array}
\right]
=
-
\exp\left( \frac{2 \pi i}{3}    \right)
\theta^3
\left[
\begin{array}{c}
\frac13 \\
1
\end{array}
\right]
+
\theta^3
\left[
\begin{array}{c}
1 \\
\frac13
\end{array}
\right],
\end{equation*}
and
\begin{equation*}
\theta^3
\left[
\begin{array}{c}
\frac13 \\
\frac53
\end{array}
\right]
=
\exp\left( \frac{ \pi i}{3}    \right)
\theta^3
\left[
\begin{array}{c}
\frac13 \\
1
\end{array}
\right]
-
\theta^3
\left[
\begin{array}{c}
1 \\
\frac13
\end{array}
\right],
\end{equation*}
which imply
\begin{equation*}
\frac{
\theta^{\prime}
\left[
\begin{array}{c}
1 \\
\frac13
\end{array}
\right]  
}
{ 
\theta
\left[
\begin{array}{c}
1 \\
\frac13
\end{array}
\right]
}
=
\frac
{
\pi \exp(\frac{2\pi i}{3})
}
{
3\sqrt{3}
}
\times
\frac
{
3
\theta^3
\left[
\begin{array}{c}
\frac13 \\
1
\end{array}
\right]
+
\sqrt{3}
i
\theta^3
\left[
\begin{array}{c}
1 \\
\frac13
\end{array}
\right]
}
{
\theta
\left[
\begin{array}{c}
\frac13 \\
1
\end{array}
\right](0,3\tau)
}. 
\end{equation*}
Therefore, 
the theorem follows from Proposition \ref{prop:1st-derivative} and Jacobi's triple product identity (\ref{eqn:Jacobi-triple}).
\end{proof}

\section{Additional product-series identities}
\label{sec:more-pro-series-level3}

\subsection{Selected theta functional formulas}

\begin{proposition}
\label{prop:theta-function-identities-level3}
{\it
For every $(z,\tau)\in\mathbb{C}\times\mathbb{H}^2,$ we have 
\begin{align}
&
\theta^2
\left[
\begin{array}{c}
\frac13 \\
1
\end{array}
\right]
\theta
\left[
\begin{array}{c}
1 \\
\frac13
\end{array}
\right](z)
\theta
\left[
\begin{array}{c}
1 \\
\frac53
\end{array}
\right](z)
+
\exp \left(\frac{\pi i}{3} \right)
\theta^2
\left[
\begin{array}{c}
1 \\
\frac13
\end{array}
\right]
\theta
\left[
\begin{array}{c}
\frac13 \\
1
\end{array}
\right](z)
\theta
\left[
\begin{array}{c}
\frac53 \\
1
\end{array}
\right](z) \notag \\
&\hspace{20mm}
-
\theta
\left[
\begin{array}{c}
\frac13 \\
\frac13
\end{array}
\right]
\theta
\left[
\begin{array}{c}
\frac13 \\
\frac53
\end{array}
\right]
\theta^2 
\left[
\begin{array}{c}
1 \\
1
\end{array}
\right](z)=0,   \label{eqn:level3-(1,1/3),(1/3,1)}
\end{align}
\begin{align}
&
\theta^2
\left[
\begin{array}{c}
\frac13 \\
\frac53
\end{array}
\right]
\theta
\left[
\begin{array}{c}
\frac13 \\
\frac13
\end{array}
\right](z)
\theta
\left[
\begin{array}{c}
\frac53\\
\frac53
\end{array}
\right](z)
+
\theta^2
\left[
\begin{array}{c}
\frac13 \\
\frac13
\end{array}
\right]
\theta
\left[
\begin{array}{c}
\frac13 \\
\frac53
\end{array}
\right](z)
\theta
\left[
\begin{array}{c}
\frac53 \\
\frac13
\end{array}
\right](z)  \notag \\
&\hspace{20mm}
-
\exp \left(\frac{2\pi i}{3} \right)
\theta
\left[
\begin{array}{c}
\frac13 \\
1
\end{array}
\right]
\theta
\left[
\begin{array}{c}
1 \\
\frac13
\end{array}
\right]
\theta^2 
\left[
\begin{array}{c}
1 \\
1
\end{array}
\right](z)=0.   \label{eqn:level3-(1/3,1/3),(1/3,5/3)}
\end{align}
}
\end{proposition}

\begin{proof}
We prove equation (\ref{eqn:level3-(1,1/3),(1/3,1)}).  
Equation (\ref{eqn:level3-(1/3,1/3),(1/3,5/3)}) can be proved in the same way. 
We first note that 
$
\dim \mathcal{F}_{2}\left[
\begin{array}{c}
0 \\
0
\end{array}
\right]
=
2,
$ 
and 
\begin{equation*}
\theta
\left[
\begin{array}{c}
1 \\
\frac13
\end{array}
\right](z,\tau)
\theta
\left[
\begin{array}{c}
1 \\
\frac53
\end{array}
\right](z,\tau)
, \,
\theta 
\left[
\begin{array}{c}
\frac13 \\
1
\end{array}
\right](z,\tau)
\theta
\left[
\begin{array}{c}
\frac53 \\
1
\end{array}
\right](z,\tau), \,
\theta^2
\left[
\begin{array}{c}
1 \\
1
\end{array}
\right](z,\tau) \in
\mathcal{F}_{2}\left[
\begin{array}{c}
0 \\
0
\end{array}
\right].  
\end{equation*}
Therefore, there exist some complex numbers, $x_1, x_2$, and $ x_3$, not all of which are zero, such that 
 \begin{align*}
 &
 x_1
\theta
\left[
\begin{array}{c}
1 \\
\frac13
\end{array}
\right](z,\tau)
\theta
\left[
\begin{array}{c}
1 \\
\frac53
\end{array}
\right](z,\tau)
+x_2
 \theta
\left[
\begin{array}{c}
\frac13\\
1
\end{array}
\right](z,\tau)
\theta
\left[
\begin{array}{c}
\frac53 \\
1
\end{array}
\right](z,\tau)
+
x_3
\theta^2
\left[
\begin{array}{c}
1 \\
1
\end{array}
\right](z,\tau)=0. 
\end{align*}
Note that in the fundamental parallelogram, 
the zero of 
$
\theta
\left[
\begin{array}{c}
1 \\
\frac13
\end{array}
\right](z),
$ 
$
\theta
\left[
\begin{array}{c}
\frac13 \\
1
\end{array}
\right](z),
$ 
or 
$
\theta
\left[
\begin{array}{c}
1 \\
1
\end{array}
\right](z)
$ 
is $z=1/3,$ $\tau/3$ or $0.$ 
Substituting $z=1/3, \tau/3,$ and $0,$ we have 
\begin{alignat*}{4}
&
&    
&
x_2
\exp \left(-\frac{\pi i}{3}\right)
\theta
\left[
\begin{array}{c}
\frac13\\
\frac13
\end{array}
\right]  
\theta
\left[
\begin{array}{c}
\frac13 \\
\frac53
\end{array}
\right]  
&    
&+x_3
\theta^2
\left[
\begin{array}{c}
1 \\
\frac13
\end{array}
\right]  
&  
&=0,  \\
&x_1
\theta
\left[
\begin{array}{c}
\frac13\\
\frac13
\end{array}
\right]  
\theta
\left[
\begin{array}{c}
\frac13 \\
\frac53
\end{array}
\right]  
&  
&
&
&+x_3
\theta^2
\left[
\begin{array}{c}
\frac13\\
1
\end{array}
\right]  
&  
&=0,  \\
&
-
x_1
\theta^2
\left[
\begin{array}{c}
1 \\
\frac13
\end{array}
\right]  
&
+&x_2
\exp \left(-\frac{ \pi i}{3}\right)
\theta^2\left[
\begin{array}{c}
\frac13 \\
1
\end{array}
\right] 
&
&
&
&=0. 
\end{alignat*}
Solving this system of equations, 
we have 
\begin{equation*}
(x_1,x_2,x_3)=\alpha
\left(
\theta^2\left[
\begin{array}{c}
\frac13 \\
1
\end{array}
\right], 
\exp \left(\frac{\pi i}{3} \right)
\theta^2\left[
\begin{array}{c}
1 \\
\frac13
\end{array}
\right], 
-
\theta\left[
\begin{array}{c}
\frac13 \\
\frac13
\end{array}
\right] 
\theta\left[
\begin{array}{c}
\frac13 \\
\frac53
\end{array}
\right] 
\right) \,\,\,   \text{for some} \,\,\alpha\in\mathbb{C}\setminus\{0\},
\end{equation*}
which proves the proposition. 
\end{proof}

\subsection{Product-series identities}

\begin{theorem}
{\it
For every $\tau \in\mathbb{H}^2,$ we have 
\begin{equation}
\label{eqn:pro-series-deg1-(1,1/3),(1/3,1)}
\frac
{
\eta^{10}(3\tau)
}
{
\eta^3( \tau) \eta^3(9 \tau)
}
=
1
+
3
\sum_{n=1}^{\infty}
(\sigma_1(n)-9 \sigma_1(n/9) ) q^n,
\end{equation}
and
\begin{equation}
\label{eqn:pro-series-deg1-(1/3,1/3),(1/3,5/3)}
\frac
{
\eta^{3}(\tau) \eta^3(9\tau)
}
{
\eta^2( 3 \tau) 
}
=
\sum_{n=0}^{\infty}
\sigma_1(3n+1)
q^{3n+1}
-
\sum_{n=0}^{\infty}
\sigma_1(3n+2)
q^{3n+2}. 
\end{equation}
where $q=\exp(2 \pi i \tau).$ 
}
\end{theorem}

\begin{proof}
By equations (\ref{eqn:relation-for-thm-1-1/3}),  (\ref{eqn:relation-for-thm-1/3-1}), and (\ref{eqn:level3-(1,1/3),(1/3,1)}), 
we derive equation (\ref{eqn:pro-series-deg1-(1,1/3),(1/3,1)}). 
Equation (\ref{eqn:pro-series-deg1-(1/3,1/3),(1/3,5/3)}) can be proved in the same way. 
\par
Comparing the coefficients of the term $z^2$ in equation (\ref{eqn:level3-(1,1/3),(1/3,1)}), 
we have 
\begin{align*}
\left\{
\theta^{\prime}
\left[
\begin{array}{c}
1 \\
1
\end{array}
\right]
\right\}^2
\frac
{
\theta
\left[
\begin{array}{c}
\frac13 \\
\frac13
\end{array}
\right]
\theta
\left[
\begin{array}{c}
\frac13 \\
\frac53
\end{array}
\right]
}
{
\theta^2
\left[
\begin{array}{c}
1 \\
\frac13
\end{array}
\right]
\theta^2
\left[
\begin{array}{c}
\frac13 \\
1
\end{array}
\right]
}
=&
\frac32
\left(
\frac
{
\theta^{\prime \prime}
\left[
\begin{array}{c}
\frac13 \\
1
\end{array}
\right]
}
{
\theta
\left[
\begin{array}{c}
\frac13 \\
1
\end{array}
\right]
}
-
\frac
{
\theta^{\prime \prime}
\left[
\begin{array}{c}
1 \\
\frac13
\end{array}
\right]
}
{
\theta
\left[
\begin{array}{c}
1 \\
\frac13
\end{array}
\right]
}
\right)  \\
=&
6 \pi i
\frac{d}{d \tau}
\log
\frac
{
\theta
\left[
\begin{array}{c}
\frac13 \\
1
\end{array}
\right]
}
{
\theta
\left[
\begin{array}{c}
1 \\
\frac13
\end{array}
\right]
}. 
\end{align*}
Therefore, 
equation (\ref{eqn:pro-series-deg1-(1,1/3),(1/3,1)}) can be obtained 
by Jacobi's triple product identity (\ref{eqn:Jacobi-triple}).
\end{proof}



\begin{thebibliography}{}
%
%

\bibitem{Ablowitz}
M. J. Ablowitz, S. Chakravarty, and H. Hahn, 
{\it Integrable systems and modular forms of level 2, }
J. Phys. A {\bf 39} (2006), 15341-15353. 

\bibitem{Berndt-Ramanujan}
B. Berndt:
Ramanujan's notebooks. Part III, Springer-Verlag, New York, 1991.



\bibitem{Berndt}
B. C. Berndt:
Number theory in the spirit of Ramanujan, 
Stud. Math. Libr., {\bf 34} Amer. Math. Soc., Providence, RI, 2006.



\bibitem{BBG}
J. M. Borwein, P. B. Borwein, and F. G. Garvan:
{\it Some cubic modular identities of Ramanujan, }
Trans. Amer. Math. Soc. {\bf 343}  (1994), 35-47. 






\bibitem{Bureau}
F. J. Bureau: 
{\it Sur des syst\'emes diff\'erentiels non lin\'eaires du troisi\'eme ordre et les \'equations diff\'erentielles non lin\'eaires associ\'ees,}
Acad. Roy. Belg. Bull. Cl. Sci. (5) {\bf 73}, (1987), 335-353. 



\bibitem{Cooper}
S. Cooper:
Ramanujan's theta functions. Springer, Cham, 2017.



\bibitem{Dickson}
L. E. Dickson: 
Modern Elementary Theory of Numbers, 
University of Chicago Press, Chicago, 1939. 








\bibitem{Farkas}
H. M. Farkas: 
{\it
Theta functions in complex analysis and number theory. 
}  
Surveys in number theory,  
Dev. Math., {\bf 17} (2008), 57-87. 



\bibitem{Farkas-Kra}
H. M. Farkas and I. Kra: 
Theta constants, Riemann surfaces and the modular group, 
AMS Grad. Studies in Math. {\bf 37} 2001. 





\bibitem{Halphen}
G. Halphen: 
{\it Sur une system d'\'equations differentielles,} 
C. R. Acad. Sci., Paris {\bf 92} (1881), 1101-1103. 


\bibitem{Huber}
T. Huber: 
{\it Differential equations for cubic theta functions,} 
Int. J. Number Theory {\bf  71}  (2011), 1945-1957. 


\bibitem{Jacobi}
C.G.J. Jacobi: 
{\it 
\"Uber die Differentialgleichung welcher die Reihen 
$1\pm 2q +2 q^4 \pm 2 q^4 +\mathrm{etc.}, \, 2q^{1/4} +2q^{9/4} +2q^{25/4} +\mathrm{etc.}$ 
Genuge leisten,}
J. Reine Angew. Math. {\bf 36} (1848), 97-112 


\bibitem{Lang}
S. Lang: 
Introduction to modular forms,
Grundlehren der mathematischen Wissenschaften, {\bf 222}, Springer-Verlag, Berlin-New York, 1976. 








\bibitem{Lomadze}
G. A. Lomadze: 
{\it Representation of numbers by sums of the quadratic forms $x_1^2+x_1x_2+x_2^2$}
Acta Arith. {\bf 54}  (1989), 9-36. 



\bibitem{Maier}
R. S. Maier: 
{\it Nonlinear differential equations satisfied by certain classical modular forms,}
Manuscripta Math. {\bf 134} (2011), 1-42. 


\bibitem{Mano}
T. Mano, Toshiyuki: 
{\it Differential relations for modular forms of level five}, 
J. Math. Kyoto Univ. {\bf 42} (2002), 41-55. 





\bibitem{Ohyama0}
Y. Ohyama: 
{\it Differential relations of theta functions. }
Osaka J. Math.  32  (1995), 431-450. 


\bibitem{Ohyama1}
Y. Ohyama: 
{\it Systems of nonlinear differential equations related to second order linear equations,}
Osaka J. Math. {\bf 33} (1996), 927-949. 

\bibitem{Ohyama2}
Y. Ohyama: 
{\it Differential equations for modular forms of level three,}
Funkcial. Ekvac. {\bf 44} (2001), 377-389. 



\bibitem{Pol}
B. van der Pol: 
{\it On a non-linear partial differential equation satisfied by the logarithm of the Jacobian theta-functions, with arithmetical applications. I, II}, 
Indagationes Math. {\bf 13} (1951), 261-271, 272-284. 


\bibitem{Ramamani}
V. Ramamani: 
{\it On some algebraic identities connected with Ramanujan's work,} 
Ramanujan International Symposium on Analysis, Macmillan of India, New Delhi (1989) 277-291




\bibitem{Rankin}
R. A. Rankin, 
{\it The construction of automorphic forms from the derivatives of a given form,}
J. Indian Math. Soc. {\bf 20} (1956), 103-116. 


\bibitem{Williams}
K. S. Williams, 
Number theory in the spirit of Liouville,  
London Math. Soc. Stud. Texts, {\bf 76} Cambridge University Press, Cambridge, 2011. 

\bibitem{WW}
E. T. Whittaker, and G. N. Watson, 
A course of modern analysis Fourth edition, Reprinted Cambridge University Press, New York 1962

 
\end{thebibliography}


\end{document}